\documentclass[10pt,reqno]{amsart}

\usepackage{amsmath, amsfonts, amssymb, latexsym, amsthm}
\usepackage[pagewise]{lineno}
\usepackage{amsmath, amsfonts, amssymb, latexsym}
\usepackage[numbers,sort&compress]{natbib}
\usepackage{mathrsfs}
\usepackage{pdfcomment}
\hypersetup{hidelinks,
	colorlinks=true,
	allcolors=black,
	pdfstartview=Fit,
	breaklinks=true
}

\everymath{\displaystyle}

\newtheorem{theorem}{Theorem}
\theoremstyle{plain}

\newtheorem{lemma}[theorem]{Lemma}
\newtheorem{definition}[theorem]{Definition}

\newtheorem{proposition}[theorem]{Proposition}
\newtheorem{corollary}[theorem]{Corollary}
\newtheorem{remark}[theorem]{Remark}

\numberwithin{equation}{section}
\numberwithin{theorem}{section}

\newcommand{\cE}{\mathcal{E}}

\newcommand{\cL}{\mathcal{L}}

\newcommand{\cF}{\mathcal{F}}

\newcommand{\E}{\mathbb{E}}
\newcommand{\R}{\mathbb{R}}
\newcommand{\F}{\mathbb{F}}
\newcommand{\N}{\mathbb{N}}
\newcommand{\bP}{\mathbb{P}} 

\def\e{\varepsilon}

\def\Ito{It\^{o} }

\begin{document}
\title
{Observability inequalities for the backward stochastic evolution equations and their applications}

\author{\sffamily Yuanhang Liu$^{1}$, Weijia Wu$^{1}$, Donghui Yang$^1$, Jie Zhong$^{2,*}$   \\
	{\sffamily\small $^1$ School of Mathematics and Statistics, Central South University, Changsha 410083, China. }\\
 {\sffamily\small $^2$ Department of Mathematics, California State University Los Angeles, Los Angeles, 90032, USA}
}
	\footnotetext[2]{Corresponding author: jiezhongmath@gmail.com }

\email{liuyuanhang97@163.com}
\email{weijiawu@yeah.net}
\email{donghyang@outlook.com}
\email{jiezhongmath@gmail.com}


\maketitle

\begin{abstract}
The present article delves into the investigation of observability inequalities pertaining to backward stochastic evolution equations. We employ a combination of spectral inequalities, interpolation inequalities, and the telegraph series method as our primary tools to directly establish observability inequalities. Furthermore, we explore three specific equations as application examples: a stochastic degenerate equation, a stochastic fourth order parabolic equation and a stochastic heat equation. It is noteworthy that these equations can be rendered null controllability with only one control in the drift term to each system.
\end{abstract}

\pagestyle{myheadings}
\thispagestyle{plain}
\markboth{OBSERVABILITY INEQUALITIES }{YUANHANG LIU, WEIJIA WU, DONGHUI YANG, AND JIE ZHONG}


\section{Introduction}
Observability inequalities play a pivotal and consequential role in exploring the stability and controllability aspects of evolution equations. These equations encompass both degenerate and non-degenerate scenarios, and a substantial body of research has been devoted to investigating observability inequalities associated with them. For instance, the degenerate parabolic equation finds application in various physical phenomena, including laminar flow, large-scale ice formations, solar radiation-climate interactions, and population genetics (for more detailed descriptions, refer to \cite{cannarsa2016global}). Consequently, control problems related to such equations have garnered significant attention, with studies on controllability and observability issues for certain degenerate parabolic equations in \cite{cannarsa2016global} and its extensive references. In \cite{cannarsa2008carleman}, the authors established the observability inequality for a class of degenerate parabolic problems, wherein the control operates on open sets. This was achieved through the utilization of Carleman estimates, employing appropriately chosen weighted functions and Hardy-type inequalities. In \cite{fragnelli2016interior}, the author investigated non-smooth, general degenerate parabolic equations in non-divergence form. The study considered scenarios involving the control on open sets, as well as instances of degeneracy and singularity within the spatial domain's interior, while incorporating Dirichlet or Neumann boundary conditions. Notably, the authors derived the observability inequality for the corresponding adjoint problem using Carleman estimates. In \cite{boutaayamou2018carleman}, the authors addressed a parabolic problem characterized by degeneracy within the spatial domain's interior. It is noted that Neumann boundary conditions were incorporated alongside the control acting upon open sets. The study primarily concentrated on two key aspects: the well-posedness of the problem itself and the formulation of Carleman estimates for the corresponding adjoint problem. As a result of their investigations, some observability inequalities were derived. For further exploration of controllability matters concerning degenerate parabolic equations, see \cite{cannarsa2005null,cannarsa2019null,du2014approximate,floridia2014approximate} and references therein. In contrast, focusing on non-degenerate equations, the study presented in \cite{zhou2012observability} delved into the realm of observability and null controllability pertaining to a specific class of one-dimensional fourth-order parabolic equations. Employing the formulation of global Carleman estimates, the author successfully derived observability inequalities for linear fourth-order parabolic equations with potentials in the one-dimensional context. In \cite{guerrero2019carleman}, the authors investigated a fourth-order parabolic equation within a bounded and smooth domain $\Omega$, where the dimension $N\geq2$. The equation was subject to homogeneous Dirichlet boundary conditions applied to both the solution and the Laplacian of the solution. An interesting study was the establishment of a Carleman inequality, enabling observation within an arbitrarily small open set $\omega$ contained within $\Omega$. This result, in turn, led to a null controllability conclusion at any given time $T > 0$ for the associated control system, with the control function acting through $\omega$.
In \cite{apraiz2014observability,phung2013observability}, the authors establish the observability inequality of the heat equation for the measurable subsets, and show the null
controllability with controls restricted over these sets.

Note that the previous research mentioned focused on deterministic equations. However, in practice, the consideration of stochastic effects requires replacing deterministic functions with stochastic processes as mathematical descriptions, leading to the formation of stochastic degenerate and non-degenerate parabolic equations. In \cite{wu2020carleman}, the authors successfully derived Carleman estimates for a backward stochastic parabolic equation characterized by weak degeneracy and a singular weight function. By combining this Carleman estimate with an approximation argument, they were able to establish the null controllability of the forward stochastic parabolic equation with weak degeneracy, involving two control functions. In \cite{baroun2022carleman}, the authors addressed the issue of null controllability concerning a specific category of stochastic degenerate parabolic equations. Initially, they established a comprehensive global Carleman estimate for a linear forward stochastic degenerate equation incorporating multiplicative noise. Leveraging this estimate, they successfully demonstrated the null controllability of the corresponding backward equation, while also achieving a partial result regarding the controllability of the forward equation. Moreover, by introducing a novel Carleman estimate for the backward equation, employing a weighted function that does not vanish at time $t = 0$, and utilizing the duality method HUM, they were able to establish the null controllability of a forward stochastic degenerate equation involving the influence of two control functions. For other stochastic degenerate parabolic equations, we refer readers to the works in \cite{liu2019carleman}. For stochastic non-degenerate parabolic equations, in \cite{lv2022null}, the authors successfully established the null controllability of linear fourth-order stochastic parabolic equations, incorporating two control functions. By employing the duality argument and a global Carleman estimate, they reduced the problem of null controllability to the issue of observability for backward fourth-order stochastic parabolic equations. In \cite{tang2009null}, the authors addressed the problem of null controllability for both forward and backward linear stochastic parabolic equations, incorporating two control functions. To establish the duality argument, they derived observability estimates for linear backward and forward stochastic parabolic equations with general coefficients. Again, these estimates were obtained through the utilization of a global Carleman estimate. Regarding the controllability and optimal control problems for other stochastic parabolic equations, please refer to reference  \cite{liu2023norm,yang2016observability}.

One of the main contributions of this paper is the direct derivation of the observability inequality for backward stochastic evolution equations, which motivated by \cite{wang2017observability}, where the authors studied the observability from measurable sets in time for an evolution equation and used the method developed by \cite{apraiz2014observability,phung2013observability} to obtain an interpolation inequality at one time point and derive the desired observability inequality. The main tool employed here is denoted as a spectral-like condition, which is different than the techniques of Carleman estimates. As some important applications, we discuss the observability from measurable sets for a stochastic degenerate equation, a stochastic fourth order parabolic equation and a stochastic heat equation, respectively. It is noteworthy that, when dealing with stochastic degenerate equations, the research process becomes inherently more intricate compared to two other non-degenerate equations. More specifically,  the equation's degeneracy introduces challenges when studying the properties of analyticity; see Section \ref{subsec:degenerate} for more details. To surmount this difficulty, we draw insights from the concepts put forth in \cite{liu2023observability}, allowing us to establish observability inequalities for stochastic degenerate equations. In these specific applications, as compared to existing literatures \cite{wu2020carleman,lv2022null,baroun2022carleman}, we demonstrate the effectiveness of employing one control in the drift term from a reduced control space. Leveraging the outcomes of spectral inequality and  interpolation inequality, we directly derive observability inequalities for these equations from measurable sets, consequently yielding null controllability results.

The rest of this paper are structured as follows. In Section 2, we prove our main
result for an abstract equation, i.e., Theorem \ref{thm:main}. In Section 3, we present some specific examples. More specifically, we discuss the observabilities from measurable sets for a stochastic degenerate equation, a stochastic fourth order parabolic equation and a stochastic heat equation in Section 3.1, Section 3.2 and Section 3.3, respectively.

\section{An abstract equation}

At first, let us introduce necessary notations.

Let $(\Omega,\mathcal{F},\left\{\mathcal{F}_t\right\}_{t\geq 0}, \bP)$ be a fixed complete filtered probability space, on which a one dimensional standard Brownian motion $\{W(t)\}_{t\geq0}$ is defined, and $\left\{\mathcal{F}_t\right\}_{t\geq 0}$ is the corresponding natural filtration, augmented by all the $\bP$-null sets in $\mathcal{F}$. We denote by $\mathbb{F}$ the progressive $\sigma$-field w.r.t. $\left\{\mathcal{F}_t\right\}_{t\geq 0}$.

Let $H$ and $U$ be  two separable  Hilbert spaces with inner products $\langle\cdot,\cdot\rangle_H$ and $\langle\cdot,\cdot\rangle_U$;
and norms $\|\cdot\|_H$ and $\|\cdot\|_U$, respectively. Fix $t \ge 0, p \in [1,\infty)$, we denote by $L^p_{\mathcal{F}_t}(\Omega;H)$  the Banach space consisting of all $H$-valued,  $\mathcal{F}_t$ measurable random
variables $X(t)$ endowed with the norm
$$
\|X(t)\|_{L^p_{\mathcal{F}_t}(\Omega;H)}=\bigg(\mathbb{E}\|X(t)\|^p_{H}\bigg)^{\frac{1}{p}}.
$$
Denote by $L^p_{\mathbb{F}}(0,T;L^q(\Omega;H))$, $p,q\in[1,\infty)$, the Banach space consisting of all $H$-valued, $\left\{\mathcal{F}_t\right\}_{t\geq 0}$-adapted processes $X$ endowed with the norm
$$
\|X(\cdot)\|_{L^p_{\mathbb{F}}(0,T;L^q(\Omega;H))}=\bigg(\int_0^T(\mathbb{E}\|X(t)\|^q_H)^{\frac{p}{q}}dt\bigg)^{\frac{1}{p}}.
$$
Denote by $L^p_{\mathbb{F}}(\Omega; L^q(0,T;H))$, $p,q\in[1,\infty)$, the Banach space consisting of all $H$-valued, $\left\{\mathcal{F}_t\right\}_{t\geq 0}$-adapted processes $X$ endowed with the norm
$$
\|X(\cdot)\|_{L^p_{\mathbb{F}}(\Omega; L^q(0,T;H))}=\bigg[\mathbb{E}\bigg(\int_0^T\|X(t)\|^q_H dt\bigg)^{\frac{p}{q}}\bigg]^{\frac{1}{p}}.
$$
Denote by $L^\infty_{\mathbb{F}}(0,T;\R)$, the Banach space consisting of all $\R$-valued, $\left\{\mathcal{F}_t\right\}_{t\geq 0}$-adapted bounded processes, with the essential supremum norm.\\
Denote by $L^q_{\mathbb{F}}(\Omega;C([0,T];H))$, $q\in[1,\infty)$, the Banach space consisting of all $H$-valued, $\left\{\mathcal{F}_t\right\}_{t\geq 0}$-adapted continuous processes $X$ endowed with the norm
$$
\|X(\cdot)\|_{L^q_{\mathbb{F}}(\Omega;C([0,T];H))}=\bigg(\mathbb{E}\|X(\cdot)\|^q_{C([0;T];H)}\bigg)^{\frac{1}{q}}.
$$

Throughout this paper, we denote by $|\cdot|$ the Lebesgue measure on $\R^n,n\geq1$. In the sequel, we shall simply denote $L^p_{\mathbb{F}}(0,T;L^p(\Omega;H))\equiv L^p_{\mathbb{F}}(\Omega; L^p(0,T;H))$ by $L^p_{\mathbb{F}}(0,T;H)$ with $p\in[1,\infty)$.

Consider the following forward controlled stochastic evolution equation:
\begin{equation}
\label{eq:main-for}
\left\{
\begin{array}{ll}
dy = A ydt + \chi_EBudt + F(t)y dW(t), & t\in(0,T),   \\[2mm]
y(0) =y_0,
\end{array}
\right.
\end{equation}
where $y_0\in L^2_{\cF_0}(\Omega;H)$, the operator $A$ generates a $C_0$-semigroup $\{S(t)\}_{t\geq0}$ on $H$ and $F\in L_{\F}^\infty(0,T;\R)$ is a given function, the control variable $u\in L^\infty_\F(0,T;L^2(\Omega; U))$, $E$ is a measurable subset of $[0, T ]$ with positive measure and denote by $\chi_E$ the characteristic function of $E$, the observation operator $B\in \cL(U,H)$ and denote by $B^*\in \cL(H,U)$ the adjoint operator of $B$.

Before we state our main theorem, let us introduce the following backward stochastic evolution equation corresponding to system \eqref{eq:main-for}:
\begin{equation}
\label{eq:main}
\left\{
\begin{array}{ll}
d z = -A zdt - F(t)Z dt + Z dW(t), & t\in(0,T),   \\[2mm]
z(T) =\eta,
\end{array}
\right.
\end{equation}
where $\eta\in L^2_{\mathcal{F}_T}(\Omega;H)$. 
We denote by $(z(\cdot;T,\eta), Z(\cdot;T,\eta))$ the solution of equation
\eqref{eq:main} given the terminal condition $\eta=z(T)$.




By the classical well-posedness result for backward stochastic evolution equations, see e.g., Theorem 4.10 in \cite{lv2021mathematical},  we know that equation (\ref{eq:main}) admits a unique solution $(z,Z)\in L^2_{\mathbb{F}}(\Omega;C([0,T];H)) \times L^2_{\mathbb{F}}(0,T;H)$.


We write
\[
0<\lambda_1\le \lambda_2\le \cdots
\]
for the eigenvalues of $-A$, and $\{e_j\}_{j\ge 1}$ for corresponding eigenfunctions which form an orthonormal basis for $H$. For each $\lambda>0$, we define
\[
\cE_\lambda f = \sum_{\lambda_j\le\lambda} \langle f,e_j\rangle_He_j,~\text{and}~
\cE^\perp_\lambda f = \sum_{\lambda_j>\lambda} \langle f,e_j\rangle_He_j, ~\text{for each}~f \in H.
\]
Consequently, we have
\[
  f = \cE_\lambda f + \cE_\lambda^\perp f.
\]

Next, we introduce the following spectral-like condition ${\bf(H)}$:

${\bf(H)}$: There are constants $\gamma\in(0,1)$ and $N>0$ such that for any $\lambda>0$,
$$
\|\cE_\lambda f\|_H\leq Ne^{N\lambda^{\gamma}}\|B^*\cE_\lambda f\|_U,~\text{for all}~f\in H.
$$

The following is our main result.
\begin{theorem}\label{thm:main}
Given $T>0$. Suppose the condition ${\bf(H)}$ holds.  Then there exists a positive constant $C=C(|E|,T,N,\gamma)$ such that the solution $z$ to the equation \eqref{eq:main} satisfies the following observability inequality
\begin{equation}\label{thm:main-ob}
\|z(0;T,\eta)\|_{L^2_{\mathcal{F}_0}(\Omega;H)} \le C \|\chi_EB^*z(\cdot;T,\eta)\|_{L^1_\F(0,T;L^2(\Omega; U))}.
\end{equation}
\end{theorem}

\begin{remark}
    It is noteworthy that in Section \ref{sec:app}, we present the explicit expression of the observation operator $B$, with the control domain being a measurable set. By leveraging Theorem \ref{thm:main}, we can establish the corresponding observability inequalities for the specific equation over a measurable set in the space-time domain. This is a new result, to  the best of our knowledge, in the stochastic case.
\end{remark}

By
linearity, it is easy to check that
\begin{align}
  & z(t; T,\cE_\lambda \eta) = \sum_{\lambda_j\le \lambda}
  z_j(t;T,\eta_j)e_j = \cE_\lambda z(t;T,\eta);\nonumber\\
& z(t; T,\cE^\bot_\lambda \eta) = \sum_{\lambda_j> \lambda}
  z_j(t;T, \eta_j)e_j = \cE^\bot_\lambda z(t;T,\eta), \label{eq:2}
\end{align}
where $\eta_j=\langle\eta,e_j\rangle_H$ and $(z_j(\cdot;T,\eta_j), Z_j(\cdot;T,\eta_j))$ is the solution of the following backward
stochastic differential equation
\begin{equation}
\label{eq:main-1}
\left\{
\begin{array}{ll}
d z_j = \lambda_j z_jdt + F(t)Z_j(t) dt + Z_j dW(t), & t\in(0,T),   \\[2mm]
z_j(T) =\eta_j.
\end{array}
\right.
\end{equation}
Set
$$
\tau=\|F\|_{L_{\F}^\infty(0,T;\R)}^2.
$$
\begin{lemma}\label{decay}
Given any $\eta$ in the space of  $L^2_{\mathcal{F}_T}(\Omega;H)$,
    we have for each $t\in [0,T]$,
  \begin{equation}
    \label{eq:decay}
    \E \|z(t ;T,\cE_\lambda^\bot \eta)\|_H^2 \le e^{(-2\lambda+\tau)(T-t)} \E \|\eta\|_H^2.
  \end{equation}
\end{lemma}
\begin{proof}
  Applying \Ito formula to
  $e^{(2\lambda-\tau)(T-t)}\|z(t;T,\cE_\lambda^\bot\eta)\|_H^2$, then integrating from $(t,T)$, we obtain
\begin{align*}
  & \|\cE^\bot_\lambda \eta\|_H^2 - e^{(2\lambda-\tau)(T-t)}
    \|z(t;T,\cE_\lambda^\bot\eta)\|_H^2\\
= &\ \int_t^T e^{(2\lambda-\tau)(T-s)} \left[2\langle z, Az\rangle_H - 2 \langle F(s)Z(s),z\rangle_H\right] ds\\
& \qquad + \int_t^T e^{(2\lambda-\tau)(T-s)} \|Z\|_{H}^2 ds + \int_t^T e^{(2\lambda-\tau)(T-s)}
    2\langle z, Z\rangle_HdW(s)\\
&\qquad  - \int_t^T e^{(2\lambda-\tau)(T-s)}
    (2\lambda-\tau) \|z(s;T,\cE_\lambda^\bot\eta)\|_H^2 ds.
\end{align*}
Taking the expectation, it follows from equality \eqref{eq:2} that
\begin{align*}
  & \E \|\cE^\bot_\lambda \eta\|_H^2 - e^{(2\lambda-\tau)(T-t)}
    \E \|z(t;T,\cE_\lambda^\bot \eta)\|_H^2\\
= &\ \E \int_t^T e^{(2\lambda-\tau(T-s)} \Big( 2
  \sum_{\lambda_j>\lambda} \lambda_j (z^j)^2-2\lambda\|z(s;T,\cE_\lambda^\bot\eta)\|_H^2 - 2F(s) \langle Z(s),z\rangle_H\\
  &\qquad+ \|Z\|_{H}^2+\tau \|z(s;T,\cE_\lambda^\bot\eta)\|_H^2\Big) ds\\
\ge & \ \E \int_t^T e^{(2\lambda-\tau)(T-s)} \Big(- 2F(s) \langle Z(s),z\rangle_H
+
  \|Z\|_{H}^2 + \tau \|z(s;T,\cE_\lambda^\bot\eta)\|_H^2\Big) ds\\
\ge & \ \E \int_t^T e^{(2\lambda-\tau)(T-s)} \Big(- 2\|F\|_{L_{\F}^\infty(0,T;\R)} \|Z\|_{H}\|z(s;T,\cE_\lambda^\bot\eta)\|_H+\|Z\|_{H}^2\\
&\qquad + \tau \|z(s;T,\cE_\lambda^\bot\eta)\|_H^2\Big) ds    \\
\ge & \ \E \int_t^T e^{(2\lambda-\tau)(T-s)}
\Big\{\big[\|Z\|_{H}-\|F\|_{L_{\F}^\infty(0,T;\R)}\|z(s;T,\cE_\lambda^\bot\eta)\|_H\big]^2\\
& \qquad + \big[\tau-\|F\|^2_{L_{\F}^\infty(0,T;\R)}\big]\|z(s;T,\cE_\lambda^\bot\eta)\|_H^2\Big\} ds\\
\ge &\ 0,
\end{align*}
which implies the inequality \eqref{eq:decay}.
\end{proof}

Next, we provide an interpolation inequality.

\begin{proposition}\label{interpolation}
  For any $\gamma\in(0,1)$ and $N>0$. Given any $\eta\in L^2_{\mathcal{F}_T}(\Omega;H)$, and $t\in [0,T)$, there
  exists a constant $K = K(T,N,\gamma)$ such that
\begin{equation}
  \label{eq:interpolation}
  \E \|z(t;T,\eta)\|_H^2 \le Ke^{K(T-t)^{-\frac{\gamma}{1-\gamma}}} \big(\E \|B^*z(t;T,\eta)\|^2_{U}\big)^{\frac{1}{2}}\big(\E \|\eta\|_H^2\big)^{\frac{1}{2}}.
\end{equation}
\end{proposition}

\begin{proof}
  Set $z = z(\cdot;T,\eta)$, then it follows from the
 spectral-like condition ${\bf(H)}$ that
\begin{align*}
  \E \|\cE_\lambda z(t)\|_H^2
& \le Ne^{N\lambda^{\gamma}}\E\|B^*\cE_\lambda z(t)\|^2_U\\
& \le Ne^{N\lambda^{\gamma}}\big(\E \|B^*z(t)\|^2_{U} +
\E\|B^*\cE^\bot_\lambda z(t)\|^2_{U}\big).
\end{align*}
Therefore, by the decay estimate
\eqref{eq:decay} we obtain that
\begin{align*}
&  \E \|z(t)\|_H^2
 = \E \|\cE_\lambda z(t)+\cE^\bot_\lambda
z(t)\|_H^2\\
& \le 2(\E \|\cE_\lambda z(t)\|_H^2+\E \|\cE^\bot_\lambda z(t)\|_H^2)\\
& \le  Ne^{N\lambda^{\gamma}}\big(\E \|B^*z(t)\|^2_{U} +
\E\|B^*\cE^\bot_\lambda z(t)\|^2_{U}\big)+ \E
\|\cE^\bot_\lambda z(t)\|_H^2\\
& \le 2Ne^{N\lambda^{\gamma}}\big(\E \| B^*z(t)\|^2_{U} + \E
\|\cE^\bot_\lambda z(t)\|_H^2\big)\\
&\le 2Ne^{N\lambda^{\gamma}}\big(\E \| B^*z(t)\|^2_{U} +
e^{(-2\lambda+\tau)(T-t)} \E \|\eta\|_H^2\big)\\
&\le  2Ne^{N\lambda^{\gamma}}e^{\tau T}\big(\E \| B^*z(t)\|^2_{U} +
e^{-2\lambda(T-t)} \E \|\eta\|_H^2\big)\\
& = 2Ne^{N\lambda^\gamma-\lambda(T-t)}e^{\tau T}\big(e^{\lambda(T-t)}\E \|B^* z(t)\|^2_{U} +
e^{-\lambda(T-t)} \E \|\eta\|_H^2\big).
\end{align*}
It is easy to verify that for all $\lambda>0$, $\gamma\in (0,1)$,
\[
 \max_{\lambda>0}\big\{N\lambda^\gamma-\lambda(T-t)\big\} \le N\bigg(\dfrac{\gamma N}{T-t}\bigg)^{\frac{\gamma}{1-\gamma}}-(\gamma N)^{\frac{1}{1-\gamma}}\bigg(\dfrac{1}{T-t}\bigg)^{\frac{\gamma}{1-\gamma}}.
\]

Hence, there exists a constant $K=K(T,N,\gamma)$ such that
\[
 \E \|z(t)\|_H^2 \le K e^{K(T-t)^{\frac{\gamma}{\gamma-1}}}\big[e^{\lambda(T-t)}\E \| B^*z(t)\|^2_{U}+
e^{-\lambda(T-t)} \E \|\eta\|_H^2\big],
\]
which is equivalent to
\begin{equation}
  \label{eq:epsilon inq}
  \E \|z(t)\|_H^2 \le K e^{K(T-t)^{\frac{\gamma}{\gamma-1}}}\big[\e^{-1}\E \| B^*z(t)\|^2_{U} +
\e \E \|\eta\|_H^2\big],\quad \forall\, \e\in (0,1).
\end{equation}
Noting that $\E \|z(t)\|_H^2 \le C \E \|\eta\|_H^2$, where $C$ is a
constant depending on $T$, we see that the inequality
\eqref{eq:epsilon inq} holds for all $\e>0$. Finally, minimizing
\eqref{eq:epsilon inq} with respect to $\e$ leads to the desired
estimate \eqref{eq:interpolation}.
\end{proof}

We are now ready to prove Theorem \ref{thm:main} by the telescoping series method (see \cite{apraiz2014observability,phung2013observability}) and Proposition \ref{interpolation}.
\begin{proof}[Proof of Theorem \ref{thm:main}]
  Let $\ell\in(0,T)$ be any Lebesgue point of $E$. Then for each constant $q\in(0,1)$ which is to be fixed later, there exists a monotone increasing sequence
$\{\ell_n\}_{n\geq1}$ in $(0,\ell)$ such that
$$\lim_{n\rightarrow +\infty}\ell_n=\ell,$$
\begin{equation}\label{eq:equiv ratio}
\ell_{n+2}-\ell_{n+1}=q(\ell_{n+1}-\ell_n),\;\forall\, n\geq1
\end{equation}
and
$$|E\cap (\ell_n,\ell_{n+1})|\geq \frac{\ell_{n+1}-\ell_n}{3},\;\forall\, n\geq1.$$
Set
\begin{equation}
\label{eq-tau}
\tau_n=\ell_{n+1}-\frac{\ell_{n+1}-\ell_n}{6},\;\forall\, n\geq1.
\end{equation}
For each $t\in (\ell_n, \tau_n)$, by the interpolation inequality
\eqref{eq:interpolation} and replacing $ T$ by $l_{n+1}$, we have
\[
\E \|z(t)\|_H^2 \le K e^{K(l_{n+1}-t)^{\frac{\gamma}{\gamma-1}}} \big(\E \|B^*z(t)\|_{U}^2\big)^{\frac{1}{2}}\big(\E \|z(\ell_{n+1})\|_H^2\big)^{\frac{1}{2}}.
\]
Since
\[
\ell_{n+1} -t \ge \ell_{n+1} -\tau_n = \frac{\ell_{n+1} -\ell_n}{6},
\]
and for some constant $C=C(T)$, $\E \|z(\ell_n)\|_H^2 \le C \E
\|z(t)\|_H^2$, there exists a constant $C = C(T,N,\gamma)$ such
that for all $n\ge 1$, and $t\in (\ell_m, \tau_n)$,
\[
\E \|z(\ell_n)\|_H^2 \le C e^{C(l_{n+1}-l_n)^{\frac{\gamma}{\gamma-1}}} \big(\E \|B^*z(t)\|_{U}^2\big)^{\frac{1}{2}}\big(\E \|z(\ell_{n+1})\|_H^2\big)^{\frac{1}{2}}.
\]
Using the Cauchy inequality with $\epsilon$, the above inequality leads to
\[
\E \|z(\ell_n)\|_H^2 \le \e^{-1} C e^{C(l_{n+1}-l_n)^{\frac{\gamma}{\gamma-1}}} \E \|B^*z(t)\|_{U}^2+ \e \E \|z(\ell_{n+1})\|_H^2.
\]
Equivalently, we have
\begin{equation}
  \label{eq:Am-inq}
  A_n \le \e^{-1} C e^{C(l_{n+1}-l_n)^{\frac{\gamma}{\gamma-1}}} G(t) + \e A_{n+1},
\end{equation}
where
\begin{equation}
  \label{eq:Am-B}
  A_n =\left(\E \|z(\ell_n)\|_H^2\right)^{\frac{1}{2}},\ G(t)=\left(\E \|B^*z(t)\|_{U}^2\right)^{\frac{1}{2}}.
\end{equation}
Noting that $\{\ell_n\}_{n\geq1}$ is a monotone increasing sequence
in $(0,\ell)$, it follows that
\begin{align*}
  |E\cap(\ell_n,\tau_n)|
& = |E \cap (\ell_n,\ell_{n+1})| - |E\cap (\tau_n,\ell_{n+1})|\\
& \ge \frac{\ell_{n+1}-\ell_{n}}{3} - \frac{\ell_{n+1}-\ell_{n}}{6}\\
& = \frac{\ell_{n+1}-\ell_{n}}{6}\\
&>0.
\end{align*}
Then integrating the inequality \eqref{eq:Am-inq} over $E\cap (\ell_n,\tau_n)$, we have that for each $\e>0$,
$$
\int_{E\cap (\ell_n,\tau_n)}A_ndt\leq \e^{-1}C e^{C(l_{n+1}-l_n)^{\frac{\gamma}{\gamma-1}}}\int_{l_n}^{\tau_n}\chi_EG(t)dt+\e\int_{E\cap (\ell_n,\tau_n)}A_{n+1}dt,
$$
which implies
$$
A_n\leq \e^{-1}C e^{C(l_{n+1}-l_n)^{\frac{\gamma}{\gamma-1}}}|E\cap(\ell_n,\tau_n)|^{-1}\int_{l_n}^{\tau_n}\chi_EG(t)dt+\e A_{n+1}.
$$
By (\ref{eq-tau}) and $|E\cap(\ell_n,\tau_n)|\geq \frac{1}{6}(\ell_{n+1}-\ell_n)$, it follows that
\[
A_n \le \e A_{n+1}+ \e^{-1} C e^{C(l_{n+1}-l_n)^{\frac{\gamma}{\gamma-1}}} \int_{\ell_{n}}^{\ell_{n+1} } \chi_E G(t) dt,
\]
where $C=C(T,N,\gamma,|E|)$.\\
Multiplying the above inequality by $\e e^{-C(l_{n+1}-l_n)^{\frac{\gamma}{\gamma-1}}}$,
and replacing $\e$ by $\sqrt{\e}$ lead to
\[
\sqrt{\e}  e^{-C(l_{n+1}-l_n)^{\frac{\gamma}{\gamma-1}}} A_n \le \e  e^{-C(l_{n+1}-l_n)^{\frac{\gamma}{\gamma-1}}} A_{n+1} + C \int_{\ell_{n}}^{\ell_{n+1} } \chi_E G(t) dt.
\]
 Finally choosing $\e =
e^{-(l_{n+1}-l_n)^{\frac{\gamma}{\gamma-1}}}$ in the above inequality, we get
\begin{align*}
  & e^{-(C+\frac{1}{2})(l_{n+1}-l_n)^{\frac{\gamma}{\gamma-1}}} A_n -
  e^{-(C+1)(l_{n+1}-l_n)^{\frac{\gamma}{\gamma-1}}} A_{n+1}\le C \int_{\ell_{n}}^{\ell_{n+1} } \chi_E G(t) dt.
\end{align*}
Now, choosing $q = \bigg(\frac{C+\frac{1}{2}}{C+1}\bigg)^{\frac{1-\gamma}{\gamma}}\in(0,1)$ in \eqref{eq:equiv ratio}, we
have
\begin{align*}
  & e^{-(C+\frac{1}{2})(l_{n+1}-l_n)^{\frac{\gamma}{\gamma-1}}} A_n -
  e^{-(C+\frac{1}{2})(l_{n+2}-l_{n+1})^{\frac{\gamma}{\gamma-1}}} A_{n+1} \le C \int_{\ell_{n}}^{\ell_{n+1} } \chi_E G(t) dt.
\end{align*}
Summing the above inequality from $n=1$ to $+\infty$, we have
\[
A_1 \le C e^{(C+\frac{1}{2})(l_2-l_1)^{\frac{\gamma}{\gamma-1}}}
\int_{\ell_1}^\ell \chi_E G(t) dt.
\]
Plugging the substitution \eqref{eq:Am-B} into the above inequality, we obtain
\begin{equation}\label{qwe}
\E \|z(\ell_1)\|_H^2 \le C e^{(C+\frac{1}{2})(l_2-l_1)^{\frac{\gamma}{\gamma-1}}}
\left(\int_{\ell_1}^\ell \chi_E \big(\E\|B^*z(t)\|_{U}^2\big)^{\frac{1}{2}} dt\right)^2,   
\end{equation}
 which, along with the fact that $\E \|z(0)\|_H^2 \le C(T) \E \|z(l_1)\|_H^2$, implies the observability inequality \eqref{thm:main-ob}. The proof is completed.
\end{proof}

\section{Applications}
\label{sec:app}
This section presents some specific examples which are under the framework of the backward stochastic evolution equation (\ref{eq:main}) and Theorem \ref{thm:main}.

\subsection{Stochastic degenerate equation}
\label{subsec:degenerate}
In this subsection, we discuss the controllability for a stochastic degenerate equation.

The set $I:=(0,1)$, spaces $H=U=L^2(I)$, $F(t) \in L^\infty_{\F}(0,T;\R)$, denote by $\chi_G$, the characteristic function of $G$, where $G\subset I$ is a  measurable subset with positive measure. 

Assume
\begin{equation}
\label{eq-a:1}
a\in C([0,1])\cap C^1((0,1]),\,\,a>0\,\, \text{on}~\,(0,1]\,\, \text{and}~\, a(0)=0.
\end{equation}
Let $A$ be an unbounded linear operator on $L^2(I)$:
$$
\left\{
\begin{array}{ll}
\mathcal{D}(A):=\{v\in H_a^1(I):av_x\in H^1(I)\},\\
Av :=  (av_x)_x,\ \forall v\in\mathcal{D}(A),
\end{array}
\right.
$$
where $H_a^1(I)$ will be defined later. By \cite{cannarsa2008carleman} and \cite{cannarsa2005null}, $A$ is the infinitesimal generator of a strongly continuous semigroup $\{S(t)\}_{t\geq0}$ on $L^2(I)$.

Consider initial state $y_0 \in L^2_{\mathcal{F}_0}(\Omega;L^2(I))$, $y$ is the $L^2(I)$-valued state variable, $u$ is the control variable and $u\in L^\infty_\F(0,T;L^2(\Omega; L^2(I)))$. The set $E$ is a measurable subset of $(0, T)$ with positive measure.
\begin{definition}
If, in addition to (\ref{eq-a:1}), the function $a$ satisfies
\begin{equation}
\label{eq-a:2}
\dfrac{1}{a}\in L^1(I),
\end{equation}
the following system
$$
\left\{
\begin{array}{ll}
d y = A ydt + \chi_E\chi_G u dt + F(t) y dW(t), & \left( x ,t\right) \in I\times(0,T),   \\[2mm]
y(1,t)=y(0,t) = 0, & t\in \left(0,T\right), \\[2mm]
y\left(x, 0\right) =y_{0}(x), &  x\in I\,.
\end{array}
\right.
$$
is called the stochastic weakly degenerate problem (SWDP).
\end{definition}

\begin{definition}
If, in addition to (\ref{eq-a:1}), the function $a$ satisfies
\begin{equation}
\label{eq-a:3}
a\in C^1([0,1]),\quad\ \text{and}~\quad \dfrac{1}{\sqrt a}\in L^1(I).
\end{equation}
the following system
$$
\left\{
\begin{array}{ll}
d y = A ydt +\chi_E\chi_G u dt + F(t) y dW(t), & \left( x,t\right) \in I\times(0,T),   \\[2mm]
y(1,t)=(a y_{x})\left( 0,t\right) =0 , & t\in \left(0,T\right), \\[2mm]
y\left(x, 0\right) =y_{0}(x), &  x\in I\,.
\end{array}
\right.
$$
is called the stochastic strongly degenerate problem (SSDP).
\end{definition}

As a result, in present subsection, we consider the following forward
stochastic degenerate equation:
\begin{equation}
\label{eq:degenerate-main}
\left\{
\begin{array}{ll}
d y = A ydt +  \chi_E\chi_G u dt + F(t) y dW(t), & \left( x,t\right) \in I\times(0,T),   \\[2mm]
y(1,t) = 0, & t\in \left(0,T\right), \\[2mm]
\mbox{and }\left\{
\begin{array}{ll}
y\left( 0,t\right) =0, & \quad \mbox{for SWDP}, \\[2mm]
(a y_{x})\left( 0,t\right) =0, & \quad \mbox{for SSDP},
\end{array}
\right.
&  t\in \left( 0,T\right), \\
\\y\left(x, 0\right) =y_{0}(x), &  x\in I\,.
\end{array}
\right.
\end{equation}
By \cite{liu2019carleman}, also see \cite{krylov2007stochastic,cannarsa2016global}, we know that the equation (\ref{eq:degenerate-main}) admits a unique solution $y\in L^2_{\mathbb{F}}(\Omega;C([0,T];L^2(I))) \cap L^2_\mathbb{F}(0,T;H_a^1(I))$, where
$$
H^1_a(I):=\left\{\ \
	\begin{array}{lll}
	\mbox{SWDP}: \bigg\{v\in L^2(I):v\, \text{absolutely continuous in}~ [0,1],\\[2mm]
 \quad\ \quad\ \quad\ \sqrt av_x\in L^2(I)\,\text{and}~\,v(0)=v(1)=0 \bigg\}.\\
\mbox{SSDP}:\,\,\, \bigg\{v\in L^2(I):v\, \text{locally absolutely continuous in}~ (0,1],\\[2mm]
 \quad\ \quad\ \quad\
\sqrt av_x\in L^2(I)\,\text{and}~\,v(1)=0 \bigg\}.
	\end{array}\right.
$$
with the norm
$$
\forall\,v\in H^1_a(I),\quad \|v\|^2_{H^1_a(I)}:=\|v\|_{L^2(I)}^2+\|\sqrt av_x\|_{L^2(I)}^2.
$$

Specifically, in this paper, we adopt the function $a(x) = x^\alpha$. Consequently, the stochastic weakly degenerate problem corresponds to $\alpha \in (0, 1)$, whereas the stochastic strongly degenerate problem corresponds to $\alpha \in [1, 2)$.

Let $\alpha\in(0,2)$ and $\sigma\in(0,1)$ be defined as
\begin{equation}\label{u}
\sigma=
\left\{
\begin{array}{lll}
\dfrac{3}{4}, &~\text{if}~ \alpha\in(0,2)\setminus\{1\},\\[3mm]
\dfrac{3}{2\gamma}  ~\text{for any}~\gamma\in(0,2),&~\text{if}~ \alpha=1.
\end{array}
\right.
\end{equation}

The main result concerns the null controllability property of forward stochastic degenerate equation (\ref{eq:degenerate-main}) as follows:
\begin{theorem}
  \label{thm:degenerate-main}
Let $\alpha\in(0,2)$ and $\sigma\in(0,1)$ be defined in (\ref{u}). Suppose $G\times E$ are measurable subsets of $I\times [0,T]$ with positive measures, then the equation (\ref{eq:degenerate-main}) is null controllable. That is,
for each initial data $y_0\in L^2_{\mathcal{F}_0}(\Omega;L^2(I))$, there
is a control $u$ in the space $L^\infty_\F(0,T;L^2(\Omega; L^2(I)))$ such that
the solution $y$ of the stochastic degenerate equation \eqref{eq:degenerate-main} satisfies $y(T) = 0$
in $I$, $\bP$-a.s. Moreover, there is a constant $C=C(T,I,\alpha,\sigma, |E|,|G|)>0$ such that the control $u$ satisfies the following
estimate
\begin{equation}
  \label{eq:control est}
  \|u\|_{L^\infty_\F(0,T;L^2(\Omega; L^2(I)))} \le C  \|y_0\|_{L^2_{\mathcal{F}_0}(\Omega;L^2(I))}.
\end{equation}
\end{theorem}
\begin{remark}
It is noteworthy that the existing research on the controllability problems associated with stochastic degenerate parabolic equation (\ref{eq:degenerate-main}) has predominantly employed the Carleman estimate technique, see \cite{wu2020carleman,liu2019carleman,baroun2022carleman}. However, the scope of these findings has been restricted to the control space $L^2_\F(0,T;L^2(I))$ and instances where the controlled region $G$ is an open subset spanning the entire time interval $[0,T]$. Moreover, these investigations have primarily concentrated on outcomes within the framework of weak degeneracy conditions, with limited exploration into the implications arising from scenarios involving strong degeneracy. In sharp contrast, our study not only encompasses control regions characterized by measurable sets in space-time, but also incorporates a control space of $L^\infty_\F(0,T;L^2(\Omega; L^2(I)))$. Furthermore, our research establishes null controllability under both strong and weak degeneracy conditions.
\end{remark}

In order to prove Theorem \ref{thm:degenerate-main}, we introduce the following backward stochastic evolution equation:
\begin{equation}
\label{eq:degenerate-adjoint}
\left\{
\begin{array}{ll}
d z = -A zdt - F(t)Z dt + Z dW(t), & \left( x,t\right) \in I\times(0,T),   \\[2mm]
z(1,t) = 0, & t\in \left(0,T\right), \\[2mm]
\mbox{and }\left\{
\begin{array}{ll}
z\left( 0,t\right) =0, & \quad \mbox{for SWDP}, \\[2mm]
(a z_{x})\left( 0,t\right) =0, & \quad \mbox{for SSDP},
\end{array}
\right.
&  t\in \left( 0,T\right), \\
\\z\left(x, T\right) =\eta, &  x\in I\,,
\end{array}
\right.
\end{equation}
where $\eta\in L^2_{\mathcal{F}_T}(\Omega;L^2(I))$.

By the standard duality argument (e.g., see \cite[Theorem 7.17]{lv2021mathematical}), Theorem \ref{thm:degenerate-main} is deduced by the following observability inequality for
the backward stochastic  degenerate equation (\ref{eq:degenerate-adjoint}).
\begin{theorem}
  \label{thm:degenerate-obs-inq}
Let $\alpha\in(0,2)$ and $\sigma\in(0,1)$ be defined in (\ref{u}). Supposed $G\times E$ are measurable subsets of $I\times (0,T)$ with positive measures. Then there exists a constant $C=C(T,I,\alpha,\sigma, |E|,|G|)$ such that the following observability inequality holds: for any $\eta\in L^2_{\mathcal{F}_T}(\Omega;L^2(I))$,
\begin{equation}
  \label{eq:degenerate-obs-inq}
  \|z(0;T,\eta)\|_{L^2_{\mathcal{F}_0}(\Omega;L^2(I))} \le C \|\chi_E\chi_Gz(\cdot;T,\eta)\|_{L^1_\F(0,T;L^2(\Omega; L^2(I)))}.
\end{equation}
\end{theorem}

In order to obtain the Theorem \ref{thm:degenerate-obs-inq}, denote by $G_0:=(\varepsilon,1)\cap G$ the subset of $I$ with sufficient small $\varepsilon>0$ such that $|G_0|\geq\frac{|G|}{2}$.

Considering the observation operator $B=\chi_{G_0}$, the characteristic function of $G_0$, we shall study the following refined forward
stochastic degenerate equation of \eqref{eq:degenerate-main}:
\begin{equation}
\label{eq:degenerate-main-re}
\left\{
\begin{array}{ll}
d y = A ydt +  \chi_{E}\chi_{G_0} u dt + F(t) y dW(t), & \left( x,t\right) \in I\times(0,T),   \\[2mm]
y(1,t) = 0, & t\in \left(0,T\right), \\[2mm]
\mbox{and }\left\{
\begin{array}{ll}
y\left( 0,t\right) =0, & \quad \mbox{for SWDP}, \\[2mm]
(a y_{x})\left( 0,t\right) =0, & \quad \mbox{for SSDP},
\end{array}
\right.
&  t\in \left( 0,T\right), \\
\\y\left(x, 0\right) =y_{0}(x), &  x\in I\,.
\end{array}
\right.
\end{equation}
It is noted that the corresponding backward stochastic equation to \eqref{eq:degenerate-main-re} is still \eqref{eq:degenerate-adjoint}.

\begin{remark}
    The reason for introducing the refined equation \eqref{eq:degenerate-main-re} stems from the difficulty in obtaining the spectral-like condition ${\bf(H)}$ when the observation operator $B$ is represented by $\chi_G$, due to the degeneracy of the original equation \eqref{eq:degenerate-main}. However, by choosing the observation operator $B$ as $\chi_{G_0}$ and excluding the degenerate points from the refined equation \eqref{eq:degenerate-main-re}, we obtain a non-degenerate system. In this context, adopting the approach proposed in \cite{apraiz2014observability} enables us to establish spectral-like condition ${\bf(H)}$, and subsequently applying the methods described in \cite{liu2023observability} leads us to Theorem \ref{thm:degenerate-obs-inq}.
\end{remark}

Now recall an important spectral inequality used later in this subsection; see Theorem 1.1 in \cite{buffe2018spectral}.
\begin{lemma}
  \label{lem:degenerate-spectral}
Let $\omega$ be an open and nonempty subset of $I$ and $\alpha\in(0,2)$, there exist constants $\sigma\in(0,1)$ defined in (\ref{u}) and $C=C(I,\alpha,\sigma)>0$ such that
\begin{equation*}
\|\cE_\lambda \xi\|_{L^2(I)}
\leq Ce^{C\lambda^\sigma}\|\cE_\lambda \xi \|_{L^2(\omega)},~\forall\, \xi\in
L^2(I), \lambda>0.
\end{equation*}
\end{lemma}

Next, we recall the following observability estimate or propagation of smallness inequality
from measurable sets:
\begin{lemma}
\label{lemma-measurable}
Let $\omega$ be a bounded domain in $\R^n(n\geq1)$ and $\tilde\omega\subset\omega$  be a measurable set of
positive measure. Let $\xi$ be an analytic function in $\omega$ satisfying
\begin{equation*}
|\partial_x^a\xi(x)|\leq M (\rho)^{-|a|}|a|!,\,\,\forall\,x\in \omega,\,\,\forall\,a\in\N^n,
\end{equation*}
where $M>0$, $\rho\in(0,1]$.
Then there are $C=C(|\tilde\omega|, \rho,|\omega| )$ and $\theta=\theta(|\tilde\omega|, \rho,|\omega| )$, $0 <\theta<1$, such that
\begin{equation}\label{M}
\|\xi\|_{L^\infty(\omega)}\leq CM^{1-\theta} \bigg( \int_{\tilde\omega}|\xi(x)|dx \bigg)^{\theta}.
\end{equation}
\end{lemma}
The estimate \eqref{M} was first established in \cite{vessella1999continuous} (see also \cite{Nadirashvili1976} and \cite{Nadirashvili1979} for other similar
results). The reader can find a simpler proof of Lemma \ref{lemma-measurable} in \cite{apraiz2013null}, building on ideas
from \cite{malinnikova2004propagation,vessella1999continuous,Nadirashvili1976}.

Lemma \ref{lem:degenerate-spectral} and Lemma \ref{lemma-measurable} imply the next result for the following deterministic degenerate equation:
\begin{equation*}
\left\{
\begin{array}{ll}
 \xi_t(x,t)-A\xi = 0, & \left( x,t\right) \in I\times(0,T),   \\[2mm]
\xi(1,t) = 0, & t\in \left(0,T\right), \\[2mm]
\mbox{and }\left\{
\begin{array}{ll}
\xi\left( 0,t\right) =0, & \quad \alpha\in(0,1), \\[2mm]
(a \xi_{x})\left( 0,t\right) =0, & \quad \alpha\in[1,2),
\end{array}
\right.
&  t\in \left( 0,T\right), \\
\\ \xi\left(x, 0\right) =\xi_{0}(x), &  x\in I\,.
\end{array}
\right.
\end{equation*}
\begin{lemma}
\label{lemma-analytic}
Set $\alpha\in(0,2)$. Let $\sigma\in(0,1)$ be defined in \eqref{u}. Then there is positive constant $C=C(I,\alpha,|G_0|)$ such that
\begin{equation}
  \label{eq:spectral}
\|\cE_\lambda \xi\|_{L^2(I)}
\leq Ce^{C\lambda^\sigma}\|\cE_\lambda \xi \|_{L^1(G_0)},~\forall\, \xi\in
L^2(I), \lambda>0.
\end{equation}
\end{lemma}
\begin{remark}
    The inequality \eqref{eq:spectral} is a specific instantiation of the spectral-like condition ${\bf(H)}$, but notable for the distinctive characteristic that the subdomain is away from the singularity point zero.
\end{remark}
\begin{proof}[Proof of Lemma \ref{lemma-analytic}]
Firstly, recalling the definition of $G_0$, we have the following cases:

(a): if $\inf G=0$, then $G_0=(\varepsilon,\sup G)\subset (\varepsilon,1)$.

(b): if $\inf G>0$, then either $G_0=(\varepsilon,\sup G)\subset (\varepsilon,1)$ or $G_0=G\subset (\varepsilon,1)$.

From either case, we know that $0\not\in\bar{G}_0$ and $G_0\subset (\varepsilon,1)$. Set $\tilde G:=(\varepsilon,1)$. Then it follows from Lemma \ref{lem:degenerate-spectral} (where $\omega$ is replaced by $\tilde G$) that we have
\begin{equation}\label{pppp}
\|\cE_\lambda \xi\|_{L^2(I)}
\leq Ce^{C\lambda^\sigma}\|\cE_\lambda \xi \|_{L^2(\tilde G)},~\forall\, \xi\in
L^2(I), \lambda>0.
\end{equation}
Let $\{e_j\}_{j\geq1}$ and $\{\lambda_j^{2\sigma}\}_{j\geq1}$ be respectively the sets of $L^2(I)$-normalized eigenfunctions and eigenvalues for $-\frac{\partial}{\partial x}\bigg(x^\alpha\frac{\partial}{\partial x}\bigg)$ with zero boundary conditions; i.e.,
\begin{equation}
\label{lemma-analytic-eq1}
\left\{
\begin{array}{ll}
\big(x^\alpha (e_j(x))_x\big)_x+\lambda^{2\sigma}_je_j(x)=0, & \text{in}~\,\, I,   \\[2mm] 
 e_j(1)= 0, \\[2mm]
\mbox{and }\left\{
\begin{array}{ll}
 e_j(0) =0, & \quad \alpha\in(0,1), \\[2mm]
(a  e_j(x)_{x})(0) =0, & \quad \alpha\in[1,2),
\end{array}
\right.
\end{array}
\right.
\end{equation}
and here, $0<\lambda_1\leq \lambda_2\leq\cdots\leq \lambda_j\leq\cdots$, and $\lim_{j\rightarrow+\infty}\lambda_j=+\infty$.

For any given $\xi\in L^2(I)$, define
\begin{equation}\label{pp}
f(x,\mu)=\sum_{\lambda_j\le\lambda} \langle\xi,e_j  \rangle_{L^2(I)}e^{\lambda_j^{\sigma}\mu} 
e_j(x),\,\,\text{for}~ \,\,x\in I,\,\,\text{and}~\,\,\mu\in\R,
\end{equation}
which, together with (\ref{lemma-analytic-eq1}), we have
\begin{equation}
\label{lemma-analytic-eq3}
\left\{
\begin{array}{ll}
\partial_\mu^2f+\frac{\partial}{\partial x}\bigg( x^\alpha\frac{\partial}{\partial x}f \bigg)=0, & \text{in}~\,\, I\times\R,   \\[2mm]
f(1)= 0,& \text{in}~\,\,  \R. \\[2mm]
\mbox{and }\left\{
\begin{array}{ll}
 f(0) =0,  \quad &\alpha\in(0,1), \\[2mm]
(a  f(x)_{x})(0) =0,  \quad &\alpha\in[1,2),
\end{array}
\right.
 & \text{in}~\,\,  \R.
\end{array}
\right.
\end{equation} 
It is easy to show that the function $x^\alpha$ is real analytic and non-degenerate in $\tilde G\times D_R(0)\subset I\times\R$ (here $D_R(0):=\{\mu\in \R:|\mu|\leq R, ~\text{for}~ R\in(0,1]\}$). By the real analytic estimates of solutions to linear elliptic equations with real analytic coefficients (cf. for instance, \cite{apraiz2013null}, Chapter 3 in \cite{john2004plane}, and Chapter 5 in \cite{morrey2009multiple}), we have that there are constants $C=C(R,\alpha)\geq1$ and $\rho=\rho(R,\alpha)\in(0,1)$ such that any solutions to equation (\ref{lemma-analytic-eq3}) satisfies,
\begin{equation}
\label{lemma-analytic-eq4}
\|\partial_x^a\partial_\mu^\gamma f\|_{L^\infty(\tilde G)}\leq\frac{C(a+\gamma)!}{(\rho)^{a+\gamma}}\bigg( \int_{\tilde G\times D_R(0)} |f|^2dxd\tau \bigg)^{\frac{1}{2}},
\end{equation}
where $a\in \N$, $\gamma\in\N$.
Thus, by \eqref{pp} and \eqref{lemma-analytic-eq4} with $\gamma=0$, $\cE_\lambda \xi = f(\cdot, 0)$  is a real-analytic function in $\tilde G$ satisfying
\begin{equation*}
\|\partial_x^a\cE_\lambda \xi\|_{L^\infty(\tilde G)}\leq\frac{Ca!}{(\rho)^a}\|f\|_{L^2\big(I\times (-R,R)\big)},~\forall\, a\in\N.
\end{equation*}
By the orthonormality of $\{e_j\}_{j\geq1}$ in $L^2(I)$ and \eqref{pp}, there is $C=C(I)$ such that
\begin{equation*}
\|f\|_{L^2\big(I\times (-R,R)\big)}\le Ce^{C\lambda^{\sigma}}\|\cE_\lambda \xi\|_{L^2(I)}.
\end{equation*}
The last two inequalities show that
\begin{equation*}
\|\partial_x^a\cE_\lambda \xi\|_{L^\infty(\tilde G)}\leq\frac{Ce^{C\lambda^{\sigma}}a!}{(\rho)^a} \|\cE_\lambda \xi\|_{L^2(I)},~\forall\, a\in\N.
\end{equation*}
In particular, $\cE_\lambda \xi$ satisfies the hypothesis in Lemma \ref{lemma-measurable} (where $\tilde\omega$ is replaced by $G_0$) with
\begin{equation*}
M=Ce^{C\lambda^{\sigma}}\|\cE_\lambda \xi\|_{L^2(I)},
\end{equation*}
and there are $C=C(|G_0|,\rho,\sigma)$ and $\theta=\theta(|G_0|,\rho,\sigma)\in(0,1)$ such that
\begin{equation}\label{ppp}
\|\cE_\lambda \xi\|_{L^\infty(\tilde G)}\leq Ce^{C\lambda^{\sigma}}\|\cE_\lambda \xi\|_{L^1(G_0)}^{\theta}\|\cE_\lambda \xi\|_{L^2(I)}^{1-\theta},
\end{equation}
which, along with \eqref{pppp}, implies that $$
\|\cE_\lambda \xi\|_{L^2(I)}
\leq Ce^{C\lambda^\sigma}\|\cE_\lambda \xi\|_{L^1(G_0)}^{\theta}\|\cE_\lambda \xi\|_{L^2(I)}^{1-\theta},~\forall\, \xi\in
L^2(I), \lambda>0.
$$ 
This deduces that the estimate \eqref{eq:spectral} holds. The proof is completed.
\end{proof}

Set
$$\tau =\|F\|^2_{L^\infty_{\F}(0,T;\R)}.$$
Next, similar to the proof of Lemma \ref{decay}, we have
\begin{corollary}\label{co-degenerate-decay}
For any $\alpha\in(0,2)$ and given any $\eta\in L^2_{\mathcal{F}_T}(\Omega;L^2(I))$,
    we have for each $t\in [0,T]$,
  \begin{equation}
    \label{eq:degenerate-decay}
    \E \|z(t ;T,\cE_\lambda^\bot \eta)\|_{L^2(I)}^2 \le e^{(-2\lambda+\tau)(T-t)} \E \|\eta\|_{L^2(I)}^2.
  \end{equation}
\end{corollary}
Then along with Corollary
\ref{co-degenerate-decay} and Lemma \ref{lemma-analytic}, we have the following  interpolation inequality:
\begin{corollary}\label{co-degenerate-inter}
  Let $\alpha\in(0,2)$ and $\sigma\in(0,1)$ be defined in (\ref{u}). Given any $\eta\in L^2_{\mathcal{F}_T}(\Omega;L^2(I))$, and $t\in [0,T)$, there
  exists
 a constant $K = K(T,I,G_0,\alpha,\sigma)$ such that
\begin{equation}
  \label{eq:degenerate-inter}
  \E \|z(t;T,\eta)\|_{L^2(I)}^2 \le Ke^{K(T-t)^{\frac{\sigma}{\sigma-1}}} \big(\E \|z(t;T,\eta)\|^2_{L^2(G_0)}\big)^{\frac{1}{2}}\big(\E \|\eta\|_{L^2(I)}^2\big)^{\frac{1}{2}}.
\end{equation}
\end{corollary}

Therefore, by employing the techniques outlined in Theorem \ref{thm:main}, in addition to the telescoping series method, we are now ready to prove Theorem \ref{thm:degenerate-obs-inq}.
\begin{proof}[Proof of Theorem \ref{thm:degenerate-obs-inq}]
Similar to the proof of \eqref{qwe}, we have 
\begin{equation*}
\E \|z(\ell_1)\|_{L^2(I)}^2 \le C e^{(C+\frac{1}{2})(l_2-l_1)^{\frac{\sigma}{\sigma-1}}}
\left(\int_{\ell_1}^\ell \chi_{E} \big(\E\|\chi_{G_0}z(t)\|_{L^2(I)}^2\big)^{\frac{1}{2}} dt\right)^2,   
\end{equation*}
which, together with the fact that $G_0\subset G$, it implies the desired estimate \eqref{eq:degenerate-obs-inq}. The proof is completed.
\end{proof}

\subsection{Stochastic fourth order parabolic equations}
In this subsection, we discuss the controllability for a stochastic fourth order parabolic equation.

Let $D$ be a bounded domain of $\R^n, n\geq1$, with boundary $\partial D$ of class $C^2$. The spaces $H=U=L^2(D)$, $F(t)\in L^\infty_{\F}(0,T;\R)$, observation operator $B=\chi_G$, denoting the characteristic function of $G$, where $G\subset D$ is a  measurable subset with positive measure. $B_R(x_0)$ stands for the ball centered at $x_0$ in $\R^n$ of radius $R$ and $B_R=B_R(0)$.

Let $A$ be an unbounded linear operator on $L^2(D)$:
$$
\left\{
\begin{array}{ll}
\mathcal{D}(A):=H^4(D)\cap H_0^2(D),\\
Av :=  -\Delta^2v,\ \forall v\in\mathcal{D}(A).
\end{array}
\right.
$$
Consider initial state $y_0 \in L^2_{\mathcal{F}_0}(\Omega;L^2(D))$, $y$ is the $L^2(D)$-valued state variable, $u$ is the control variable and $u\in L^\infty_\F(0,T;L^2(\Omega; L^2(D)))$. The set $E$ is a measurable subset of $[0, T ]$ with positive measure.

In present subsection, we consider the following forward stochastic fourth order parabolic equation:
\begin{equation}
\label{eq:fourth-main}
\left\{
\begin{array}{ll}
d y = A ydt + \chi_E\chi_G u dt + F(t)y dW(t), & \mbox{in}\,\, D\times(0,T),   \\[2mm]
y=\Delta y = 0, & \mbox{on}\,\, \partial D\times(0,T), \\[2mm]
y(0) =y_{0}, &  \mbox{in}\,\, D\,.
\end{array}
\right.
\end{equation}

By the classical well-posedness result for stochastic evolution equations, we know that equation (\ref{eq:fourth-main}) admits a unique weak solution $y\in L^2_{\mathbb{F}}(\Omega;C([0,T];L^2(D)))\cap L^2_\F(0,T;H^2(D)\cap H_0^1(D))$, see \cite{lv2022null}.

The main result concerns the null controllability property of the equation (\ref{eq:fourth-main}) as follows:
\begin{theorem}
  \label{thm:fourth-main}
Let $D$ be a bounded domain of $\R^n, n\geq1$, with boundary $\partial D$ of class $C^2$. Let $x_0\in D$ and $R\in(0,1]$ such that $B_{4R}(x_0)\subseteq D$. Suppose $G\subset D$ is a measurable subset with positive measure, contained in $B_{R}(x_0)\subseteq D$. Supposed $E$ is measurable subset of $[0,T]$ with positive measure. Then the equation \eqref{eq:fourth-main} is null controllable at any time $T>0$. That is, for each initial data $y_0\in L^2_{\mathcal{F}_0}(\Omega;L^2(D))$, there
is a control $u\in L^\infty_\F(0,T;L^2(\Omega; L^2(D)))$ such that
the solution $y$ of the stochastic fourth order parabolic equations \eqref{eq:fourth-main} satisfies $y(T) = 0$
in $D$, $\bP$-a.s. Moreover, there exists a constant $C=C(T,D,|E|,|G|,R)>0$  and the control $u$ satisfies the following
estimate
\begin{equation}
  \label{eq:fourth-control est}
  \|u\|_{L^\infty_\F(0,T;L^2(\Omega; L^2(D)))} \le C  \|y_0\|_{L^2_{\mathcal{F}_0}(\Omega;L^2(D))}.
\end{equation}
\end{theorem}
\begin{remark}
It is noteworthy that the authors in \cite{lv2022null} examine the controllability outcomes of a stochastic fourth order parabolic equation with the influence of two controls in the space of $L^2_\F(0,T;L^2(G))\times L^2_\F(0,T;L^2(D))$, over the time interval $(0, T)$. In contrast, our stochastic fourth order parabolic equation (\ref{eq:fourth-main}) not only necessitates only one control in the space of $L^\infty_\F(0,T;L^2(\Omega; L^2(D)))$ in the drift term, but also attains null controllability results over measurable sets.
\end{remark}

Similar discussion to subsection \ref{subsec:degenerate}, we introduce the following backward stochastic equation:
\begin{equation}
\label{eq:fourth-adjoint}
\left\{
\begin{array}{ll}
d z = -A zdt - F(t)Z dt + Z dW(t), & \mbox{in}\,\, D\times(0,T),   \\[2mm]
z=\Delta z = 0, & \mbox{on}\,\, \partial D\times(0,T), \\[2mm]
z(T) =\eta, &  \mbox{in}\,\, D\,,
\end{array}
\right.
\end{equation}
where $\eta\in L^2_{\mathcal{F}_T}(\Omega;L^2(D))$.

By standard duality argument, see Theorem 7.17 in \cite{lv2021mathematical}, Theorem \ref{thm:fourth-main} is deduced by the following observability inequality for
the backward stochastic  equation (\ref{eq:fourth-adjoint}):
\begin{theorem}
  \label{thm:fourth-obs-inq}
Let $D$ be a bounded domain of $\R^n, n\geq1$, with boundary $\partial D$ of class $C^2$. Let $x_0\in D$ and $R\in(0,1]$ such that $B_{4R}(x_0)\subseteq D$. Suppose $G\subset D$ is a measurable subset with positive measure, contained in $B_{R}(x_0)\subseteq D$. Supposed $E$ is measurable subset of $[0,T]$ with positive measure. Then there exists a constant $C=C(T,D,|E|,|G|,R)>0$ such that the following observability inequality holds: for any $\eta\in L^2_{\mathcal{F}_T}(\Omega;L^2(D))$,
\begin{equation}
  \label{eq:fourth-obs-inq}
  \|z(0;T,\eta)\|_{L^2_{\mathcal{F}_0}(\Omega;L^2(D))} \le C \|\chi_E\chi_Gz(\cdot;T,\eta)\|_{L^1_\F(0,T;L^2(\Omega; L^2(D)))}.
\end{equation}
\end{theorem}

Now recall an important spectral inequality used later in this subsection; see Section 3 in \cite{buffe2018spectral}, also see \cite{apraiz2013null,le2019spectral}.
\begin{lemma}
  \label{lem:fourth-spectral-1}
  For each $R\in(0,1]$. There exists constant $C=C(D,R)>0$ such that when $B_{4R}(x_0)\subset D$,
\begin{equation}
  \label{eq:fourth-spectral-1}
\|\cE_\lambda \xi\|_{L^2(D)}
\leq Ce^{C\lambda^{\frac{1}{4}}}\|\cE_\lambda \xi \|_{L^2(B_R(x_0))},\  \forall\, \xi\in
L^2(D), \ \lambda>0.
\end{equation}
\end{lemma}
Similar to Lemma \ref{lemma-measurable}, we have the following observability estimate or propagation of smallness inequality
from measurable sets:
\begin{lemma}
\label{lemma-measurable-rr}
Let $\xi:\R^n(n\geq1)\supset B_{2R}\rightarrow\R$ be an analytic function in $B_{2R}$ satisfying
\begin{equation*}
|\partial_x^a\xi(x)|\leq M (\rho R)^{-|a|}|a|!,\,\,\forall\,x\in B_{2R},\,\,\forall\,a\in\N^n,
\end{equation*}
where $M>0$, $\rho\in(0,1]$. Set $\omega \subset B_{R}$ be a measurable set with positive measure. Then there are $C=C(\frac{|\omega|}{|B_{R}|}, \rho )$ and $\theta=\theta(\frac{|\omega|}{|B_{R}|}, \rho )$, $0 <\theta<1$, such that
\begin{equation}\label{M-1}
\|\xi\|_{L^\infty(B_{R})}\leq CM^{1-\theta} \bigg( \int_{\omega}|\xi(x)|dx \bigg)^{\theta}.
\end{equation}
\end{lemma}
Lemma \ref{lem:fourth-spectral-1} and Lemma \ref{lemma-measurable-rr} imply the following spectral-like condition ${\bf(H)}$ for the following deterministic fourth order parabolic equation.
\begin{equation*}
\left\{
\begin{array}{ll}
\xi_t-A\xi = 0, & \mbox{in}\,\, D\times(0,T),   \\[2mm]
\xi=\Delta \xi= 0, & \mbox{on}\,\, \partial D\times(0,T), \\[2mm]
\xi(0) =\xi_{0}, &  \mbox{in}\,\, D\,.
\end{array}
\right.
\end{equation*}
\begin{lemma}
  \label{lem:fourth-spectral}
  Let $D$ be a bounded domain of $\R^n, n\geq1$, with boundary $\partial D$ of class $C^2$. Let $x_0\in D$ and $R\in(0,1]$ such that $B_{4R}(x_0)\subseteq D$. Suppose $G\subset D$ is a measurable subset with positive measure, contained in $B_{R}(x_0)\subseteq D$. Then there exists a positive constant $C=C(D,R,\frac{|G|}{|B_{R}|})$ such that
\begin{equation}
  \label{eq:fourth-spectral}
\|\cE_\lambda \xi\|_{L^2(D)}
\leq Ce^{C\lambda^{\frac{1}{4}}}\|\cE_\lambda \xi \|_{L^1(G)},\  \forall\, \xi\in
L^2(D), \ \lambda>0.
\end{equation}
\end{lemma}
\begin{proof}
The proof is very similar to the proof of Theorem 5 in \cite{apraiz2014observability} by using Lemma \ref{lem:fourth-spectral-1} and Lemma \ref{lemma-measurable-rr}. Hence, we omit the details.
\end{proof}

Set $$\tau =\|F\|^2_{L^\infty_{\F}(0,T;\R)}.$$

Similar to the proof of Lemma \ref{decay} and Proposition \ref{interpolation} using Lemma \ref{lem:fourth-spectral}, we have the following decay estimate and interpolation inequality, respectively.
\begin{corollary}\label{co-fourth-decay}
Given any $\eta$ in the space of  $L^2_{\mathcal{F}_T}(\Omega;L^2(D))$,
    we have for each $t\in [0,T]$,
  \begin{equation}
    \label{eq:fourth-decay}
    \E \|z(t ;T,\cE_\lambda^\bot \eta)\|_{L^2(D)}^2 \le e^{(-2\lambda+\tau)(T-t)} \E \|\eta\|_{L^2(D)}^2.
  \end{equation}
\end{corollary}

\begin{corollary}\label{co-fourth-inter}
Given any $\eta\in L^2_{\mathcal{F}_T}(\Omega;L^2(D))$, and $t\in [0,T)$, there exists a constant $K = K(T,D,G)$ such that
\begin{equation}
  \label{eq:fourth-inter}
  \E \|z(t;T,\eta)\|_{L^2(D)}^2 \le Ke^{K(T-t)^{-\frac{1}{3}}} \big(\E \|z(t;T,\eta)\|^2_{L^2(G)}\big)^{\frac{1}{2}}\big(\E \|\eta\|_{L^2(D)}^2\big)^{\frac{1}{2}}.
\end{equation}
\end{corollary}

Therefore, by employing the techniques outlined in the proof of Theorem \ref{thm:main}, in addition to the telescoping series method, we obtain Theorem \ref{thm:fourth-obs-inq}.

\subsection{Stochastic heat equation}
In this subsection, we discuss the controllability for a stochastic heat equation.

Let $D$ be a bounded domain of $\R^d, d\geq1$, with boundary $\partial D$ of class $C^2$. The spaces $H=U=L^2(D)$, $F(t)\in L^\infty_{\F}(0,T;\R)$. Let $x_0\in D$ and $R\in(0,1]$ such that $B_{4R}(x_0)\subseteq D$. Set observation operator $B=\chi_G$, denotes the characteristic function of $G$, where $G\subset D$ is a measurable subset with positive measure, contained in $B_{R}(x_0)\subseteq D$.

Let $A$ be an unbounded linear operator on $L^2(D)$:
$$
\left\{
\begin{array}{ll}
\mathcal{D}(A):=H^2(D)\cap H_0^1(D),\\
Av :=  \Delta v,\ \forall v\in\mathcal{D}(A).
\end{array}
\right.
$$

Consider initial state $y_0 \in L^2_{\mathcal{F}_0}(\Omega;L^2(D))$, $y$ is the $L^2(D)$-valued state variable, $u$ is the control variable and $u\in L^\infty_\F(0,T;L^2(\Omega; L^2(D)))$. The set $E$ is a measurable subset of $[0, T ]$ with positive measure.

In present subsection, we consider the following forward stochastic heat equation:
\begin{equation}
\label{eq:heat-main}
\left\{
\begin{array}{ll}
d y = A ydt + \chi_E\chi_G u dt + F(t) y dW(t), & \mbox{in}\,\, D\times(0,T),   \\[2mm]
y = 0, & \mbox{on}\,\, \partial D\times(0,T), \\[2mm]
y(0) =y_{0}, &  \mbox{in}\,\, D\,.
\end{array}
\right.
\end{equation}

By the classical well-posedness result for stochastic evolution equations, we know that equation (\ref{eq:heat-main}) admits a unique weak solution $y\in L^2_{\mathbb{F}}(\Omega;C([0,T];L^2(D)))\cap L^2_\F(0,T; H_0^1(D))$.

The main result concerns the null controllability property of forward stochastic heat equations (\ref{eq:heat-main}) as follows:
\begin{theorem}
  \label{thm:heat-main}
The equation \eqref{eq:heat-main} is null controllable. That is, for each initial data $y_0\in L^2_{\mathcal{F}_0}(\Omega;L^2(D))$, there is a control $u\in L^\infty_\F(0,T;L^2(\Omega; L^2(D)))$ such that
the solution $y$ of the stochastic heat equations \eqref{eq:heat-main} satisfies $y(T) = 0$
in $D$, $\bP$-a.s. Moreover, there exists a constant $C=C(T,D,|E|,|G|,R)>0$  and the control $u$ satisfies the following
estimate
\begin{equation}
  \label{eq:heat-control est}
  \|u\|_{L^\infty_\F(0,T;L^2(\Omega; L^2(D)))} \le C  \|y_0\|_{L^2_{\mathcal{F}_0}(\Omega;L^2(D))}.
\end{equation}
\end{theorem}

Similar discussion to Subsection \ref{subsec:degenerate}, we introduce the following backward stochastic equation:
\begin{equation}
\label{eq:heat-adjoint}
\left\{
\begin{array}{ll}
d z = -A zdt -F(t)Z dt + Z dW(t), & \mbox{in}\,\, D\times(0,T),   \\[2mm]
z = 0, & \mbox{on}\,\, \partial D\times(0,T), \\[2mm]
z(T) =\eta, &  \mbox{in}\,\, D\,,
\end{array}
\right.
\end{equation}
where $\eta\in L^2_{\mathcal{F}_T}(\Omega;L^2(D))$.

By the standard duality argument, see Theorem 7.17 in \cite{lv2021mathematical}, Theorem \ref{thm:heat-main} is deduced by the following observability inequality for
the backward stochastic  equation (\ref{eq:heat-adjoint}):
\begin{theorem}
  \label{thm:heat-obs-inq}
There exists a constant $C=C(T,D,|E|,|G|,R)>0$ such that the following observability inequality holds: for any $\eta\in L^2_{\mathcal{F}_T}(\Omega;L^2(D))$,
\begin{equation}
  \label{eq:heat-obs-inq}
  \|z(0;T,\eta)\|_{L^2_{\mathcal{F}_0}(\Omega;L^2(D))} \le C \|\chi_E\chi_Gz(\cdot;T,\eta)\|_{L^1_\F(0,T;L^2(\Omega; L^2(D)))}.
\end{equation}
\end{theorem}

From [\cite{apraiz2014observability}, Theorem
5], also see Lemma 2.1 in \cite{yang2016observability}, we have the following spectral-like condition ${\bf(H)}$:
\begin{lemma}
  \label{lem:heat-spectral}
  Let $D$ be a bounded domain of $\R^d, d\geq1$, with boundary $\partial D$ of class $C^2$. Let $x_0\in D$ and $R\in(0,1]$ such that $B_{4R}(x_0)\subseteq D$. Suppose $G\subset D$ is a measurable subset with positive measure, contained in $B_{R}(x_0)\subseteq D$. Then there exists a positive constant $C=C(D,R,|G|)$ such that
\begin{equation}
  \label{eq:heat-spectral}
\|\cE_\lambda \xi\|_{L^2(D)}^2
\leq Ce^{C\lambda^{\frac{1}{2}}}\|\cE_\lambda \xi \|^2_{L^2(G)},\  \forall\, \xi\in
L^2(D), \ \lambda>0.
\end{equation}
\end{lemma}
Set $$\tau =\|F\|^2_{L^\infty_{\F}(0,T;\R)}.$$

Similar to the proof of Lemma \ref{decay} and Proposition \ref{interpolation} using Lemma \ref{lem:heat-spectral}, we have the following decay estimate and interpolation inequality, respectively.
\begin{corollary}\label{co-heat-decay}
Given any $\eta$ in the space of  $L^2_{\mathcal{F}_T}(\Omega;L^2(D))$,
    we have for each $t\in [0,T]$,
  \begin{equation}
    \label{eq:heat-decay}
    \E \|z(t ;T,\cE_\lambda^\bot \eta)\|_{L^2(D)}^2 \le e^{(-2\lambda+\tau)(T-t)} \E \|\eta\|_{L^2(D)}^2.
  \end{equation}
\end{corollary}

\begin{corollary}\label{co-heat-inter}
Given any $\eta\in L^2_{\mathcal{F}_T}(\Omega;L^2(D))$, and $t\in [0,T)$, there exists a constant $K = K(T,D,|G|,R)$ such that
\begin{equation}
  \label{eq:heat-inter}
  \E \|z(t;T,\eta)\|_{L^2(D)}^2 \le Ke^{K(T-t)^{-1}} \big(\E \|z(t;T,\eta)\|^2_{L^2(G)}\big)^{\frac{1}{2}}\big(\E \|\eta\|_{L^2(D)}^2\big)^{\frac{1}{2}}.
\end{equation}
\end{corollary}

Therefore, by employing the techniques outlined in Theorem \ref{thm:main}, in addition to the telescoping series method, we obtain Theorem \ref{thm:heat-obs-inq}.

\section*{Acknowledgments}
This work is supported by the National Natural Science Foundation of China, the Science Technology Foundation of Hunan Province.

\bibliographystyle{abbrvnat}
\bibliography{ref.bib}

\end{document}